\documentclass[12pt]{article} 
\usepackage{amssymb}
\usepackage{amsmath,bm}
\usepackage{graphics}
\usepackage{color}
\usepackage{xcolor}
\usepackage{xypic}
\usepackage{tikz}
\usepackage[utf8]{inputenc}                    
\usepackage{microtype}                         
\usepackage{mathtools, amsthm, amssymb, eucal} 
\usepackage[normalem]{ulem}                    
\usepackage{mdframed}
\usepackage{verbatim}
\usepackage{booktabs} 
\usepackage{mathrsfs}
\usepackage{titling}
\usepackage{xypic}
\usepackage[paper=letterpaper, margin=1in, headsep=20pt]{geometry}
\usepackage{hyperref}
\usepackage{enumerate}
\usepackage{multicol}
    \theoremstyle{plain}
    \newtheorem{thm}{Theorem}[section]
    
    \newtheorem{prop}[thm]{Proposition}
    \newtheorem{cor}[thm]{Corollary}

    \theoremstyle{definition}
    \newtheorem{defn}{Definition}[section]

    \theoremstyle{remark}
    \newtheorem{rem}{Remark}[section]

\input prepictex   \input pictex    \input postpictex

\def\mput #1 #2/{\put{$\bullet_{#1}$} [tl] <-1mm,1mm> at #2}

\begin{document}

\centerline{\textbf{\Large A 4-fold categorical equivalence}}
\medskip
\centerline{\large Ray Maresca}

\begin{abstract}

\noindent
In this note, we will illuminate some immediate consequences of work done by Reineke in [\ref{ref: every projective variety is a}] that may prove to be useful in the study of elliptic curves. In particular, we will construct an isomorphism between the category of smooth projective curves with a category of quiver grassmannians. We will use this to provide a 4-fold categorical equivalence between a category of quiver grassmannians, smooth projective curves, compact Riemann surfaces and fields of transcendence degree 1 over $\mathbb{C}$. We finish with noting that the category of elliptic curves is isomorphic to a category of quiver grassmannians, whence providing an analytic group structure to a class of quiver grassmannians.
\end{abstract}

\section{Introduction}

\indent

It is well known that there is a three-fold equivalence between the categories of compact Riemann surfaces, fields of transcendence degree 1 over $\mathbb{C}$, and smooth projective curves [\ref{ref: harshorne}] and [\ref{ref: elliptic curves reference}]. A more recent development is the notion of quiver grassmannians, first introduced by Schofield in [\ref{ref: quiver grassmannians intro}]. Since their introduction, they have become a popular topic of research. It has been known that quiver grassmannians are projective varieties, but just how much projective geometry is captured by quiver grassmannians was unclear until the early 2010's. A famous result of Hille [\ref{ref: Hille generalized Kronecker}], Huisgen-Zimmermann [\ref{ref: H-Z}], and Reineke [\ref{ref: every projective variety is a}], is that all projective varieties can be realized as quiver grassmannians for some wild acyclic quiver $Q$. Actually, even more is true. Expanding on his work in [\ref{ref: generalized Kronecker}] in which he proved the result for a generalized Kronecker quiver, Ringel showed in [\ref{ref: generalizaion of Reineke result}], the incredible result that given \textit{any} wild quiver $Q$, we can realize \textit{all} projective varieties as the quiver grassmannian of a suitable $Q$-representation. It may be interesting to ask, is there is a `best' quiver with which to study projective varieties, and if not, which quivers are `better' in which circumstances? Another natural question to ask is, can we restrict the quiver $Q$ and still get a similar result? Ringel showed in [\ref{ref: ringel Auslander varieties}] that the answer to this question is partially yes. Namley, Ringel showed that for a (controlled) wild algebra, any projective variety can be realized as an Auslander variety, but not necessarily as a quiver Grassmannian.\\

In this note, we will use the construction given by Reineke in [\ref{ref: every projective variety is a}] to define a functor from the category of smooth projective curves to a subcategory of quiver grassmannians. We will show that this functor is an isomorphism of categories, which will ultimately yield a four-fold categorical equivalence. We finish this note with some immediate consequences regarding elliptic curves.   \\ 

\noindent
\textbf{Acknowledgments:} The author would like to thank Rahul Krishna for the useful conversations. He also thanks an anonymous referee for pointing out several important references.

\section{Preliminaries}

To establish the equivalence, we will first recall some definitions.

\subsection{Projective Varieties}

\indent

Following [\ref{ref: elliptic curves reference}], let $\Bbbk$ denote a perfect field. We begin by recalling that \textbf{projective} $n$-space over a field $\Bbbk$ is defined as  $\mathbb{P}^n_{\overline{\Bbbk}} = \mathbb{P}^n = {\overline{\Bbbk}^{n+1} \over \sim }$ where $(z_0, \dots , z_n) \sim (z'_0, \dots , z'_n)$ if and only if there exists $\lambda \in \overline{\Bbbk}^*$ such that $(\lambda z_0, \dots , \lambda z_n) = (z'_0, \dots , z'_n)$ and $\overline{\Bbbk}$ denotes the algebraic closure of $\Bbbk$. Denote by $[z_0,\dots, z_n]$ the \textbf{class} of $(z_0, \dots , z_n)$ under the aforementioned quotient map. Let $R = \overline{\Bbbk}[x_0,\dots,x_n]$ be the polynomial ring in $n+1$ variables over $\overline{\Bbbk}$. A polynomial $P\in R$ is \textbf{homogeneous} of degree $d$ if $P(\lambda x) = \lambda^d P(x)$ for all $\lambda \in \overline{\Bbbk}^*$. An ideal $I \subset R$ is \textbf{homogeneous} if $I$ is generated by homogeneous polynomials. A \textbf{projective algebraic set} is some subset of $\mathbb{P}^n$ of the form $V(I) = \{[x_0,\dots, x_n] \in \mathbb{P}^n : P(x) = 0 \, \, \text{for all homogeneous} \, \, P\in I \subset R\}$ where $I$ is a homogeneous ideal. A \textbf{projective algebraic variety} is $V(I)$ for $I$ a prime homogeneous ideal of $R$.\\ 

We define the \textbf{field of rational functions} of $\mathbb{P}^N$ by $\overline{\Bbbk}(\mathbb{P}^N) = \{{f \over g}\}$ where $f,g \in R$, $g \neq 0$ and both $f$ and $g$ are homogeneous of the same degree. The \textbf{field of rational functions} of a projective variety $V\subset \mathbb{P}^N$ is defined as $\overline{\Bbbk}(V) = {\overline{\Bbbk}(\mathbb{P}^N) \over \sim}$ where ${f_1 \over g_1} \sim {f_2 \over g_2}$ if and only if $f_1g_2 - f_2g_1 \in I(V)$ where $I(V) = \{ \text{homogeneous} \, \, P\in R : P(x) = 0 \, \, \text{for all} \, \, x \in \overline{\Bbbk}^{n+1}\}$. We define a \textbf{rational map} between a projective variety $V\subset \mathbb{P}^N$ and $\mathbb{P}^M$ as the data of $M+1$ elements of $\overline{\Bbbk}(V)$. A \textbf{rational map} $V \rightarrow V' \subset \mathbb{P}^M$ is a rational map $V \rightarrow \mathbb{P}^M$ such that $[f_0, \dots , f_M](x) \in V'$ for all $x\in V$ for which $[f_0 , \dots , f_M](x)$ is defined. A rational map between varieties is called a \textbf{morphism} if it is defined everywhere. \\

The \textbf{dimension} of a projective variety is the transcendence degree of $\overline{\Bbbk}(V)$ over $\overline{\Bbbk}$. Projective varieties $V\subset \mathbb{P}^2$ of dimension one are called \textbf{projective curves}. A projective variety is called \textbf{non-singular}, or \textbf{smooth}, if the dimension of its tangent space equals its dimension at every point. For more on projective algebraic geometry, see [\ref{ref: elliptic curves reference}] and [\ref{ref: harshorne}]. The following $3$-fold categorical equivalence is well known. \\

\begin{thm} \label{thm: 3- fold Categorical equivalence}
The following three categories are equivalent:
\begin{enumerate}
    \item Compact connected Riemann Surfaces with holomorphic maps.
    \item Field extensions of transcendence degree one over $\mathbb{C}$ with field morphisms.
    \item Smooth projective curves in $\mathbb{P}^2_{\mathbb{C}}$ with morphisms of varieties. \hfill $\square$
\end{enumerate}
\end{thm}

\subsection{Quiver Grassmannians}

\indent

A \textbf{quiver} $Q$ is a directed graph. More formally, it is a $4$-tuple $Q = (Q_0,Q_1,s,t)$ where $Q_0$ is the \textbf{set of vertices}, $Q_1$ is the \textbf{set of arrows}, and $s$ and $t$ are maps that assign to each vertex a \textbf{starting} and \textbf{terminal} point respectively. For a field $\Bbbk$ that is usually taken to be algebraically closed but need not be, a \textbf{representation} $V$ of a quiver $Q$ is an assignment of a $\Bbbk$-vector space $V_i$ for each $i\in Q_0$ and a vector space morphism $\phi_{\alpha}:V_i \rightarrow V_j$ for each $\alpha \in Q_1$ such that $s(\alpha) = i$ and $t(\alpha) = j$. A \textbf{subrepresentation} $M = (M_i,\psi_{\alpha})$ of a representation $V = (V_i,\phi_{\alpha})$ is a representation of $Q$ such that $M_i \subset V_i$ is a sub vector space for all vertices $i$, $\psi_{\alpha}$ is the restriction of $\phi_{\alpha}$ to $M_{s(\alpha)}$, and $\psi_{\alpha}(M_i) \subset M_j$ for all arrows $\alpha:i \rightarrow j \in Q_1$. In other words, a subrepresentation is a collection of subspaces that are compatible with the morphisms defining the parent representation. The \textbf{dimension vector} of a representation is $\textbf{dim}V = ($dim$V_1, \dots , $dim$V_{|Q_0|})$. We call a representation $V$ \textbf{finite dimensional} if $V_i$ is a finite dimensional vector space for all $i\in Q_0$. \\

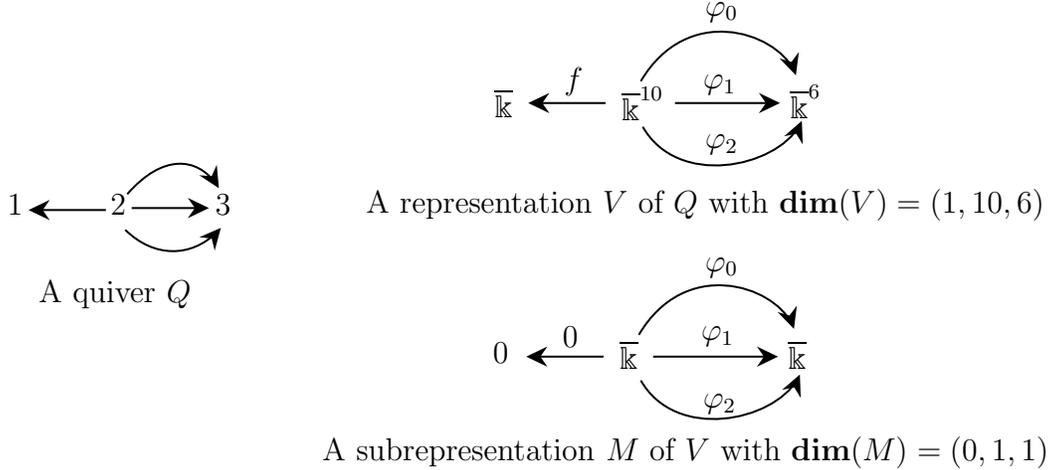
\begin{figure}
    \centering
\tikzset{every picture/.style={line width=0.75pt}} 

\begin{tikzpicture}[x=0.75pt,y=0.75pt,yscale=-1,xscale=1]

\draw    (142.22,113.01) -- (106.22,113.01) ;
\draw [shift={(103.22,113.01)}, rotate = 360] [fill={rgb, 255:red, 0; green, 0; blue, 0 }  ][line width=0.08]  [draw opacity=0] (10.72,-5.15) -- (0,0) -- (10.72,5.15) -- (7.12,0) -- cycle    ;
\draw    (153,105) .. controls (171.16,84.22) and (188.24,86.64) .. (197.66,99.63) ;
\draw [shift={(199.22,102.01)}, rotate = 239.04] [fill={rgb, 255:red, 0; green, 0; blue, 0 }  ][line width=0.08]  [draw opacity=0] (10.72,-5.15) -- (0,0) -- (10.72,5.15) -- (7.12,0) -- cycle    ;
\draw    (155.22,112.01) -- (191.22,112.01) ;
\draw [shift={(194.22,112.01)}, rotate = 180] [fill={rgb, 255:red, 0; green, 0; blue, 0 }  ][line width=0.08]  [draw opacity=0] (10.72,-5.15) -- (0,0) -- (10.72,5.15) -- (7.12,0) -- cycle    ;
\draw    (152,123) .. controls (166.46,137.26) and (185.23,137.07) .. (198.22,124.15) ;
\draw [shift={(200.22,122.01)}, rotate = 130.91] [fill={rgb, 255:red, 0; green, 0; blue, 0 }  ][line width=0.08]  [draw opacity=0] (10.72,-5.15) -- (0,0) -- (10.72,5.15) -- (7.12,0) -- cycle    ;
\draw    (394.22,59.01) -- (358.22,59.01) ;
\draw [shift={(355.22,59.01)}, rotate = 360] [fill={rgb, 255:red, 0; green, 0; blue, 0 }  ][line width=0.08]  [draw opacity=0] (10.72,-5.15) -- (0,0) -- (10.72,5.15) -- (7.12,0) -- cycle    ;
\draw    (412,48) .. controls (435.25,8.65) and (479.5,22.74) .. (489.19,42.51) ;
\draw [shift={(490.22,45.01)}, rotate = 251.57] [fill={rgb, 255:red, 0; green, 0; blue, 0 }  ][line width=0.08]  [draw opacity=0] (10.72,-5.15) -- (0,0) -- (10.72,5.15) -- (7.12,0) -- cycle    ;
\draw    (429.22,59.01) -- (479.22,59.01) ;
\draw [shift={(482.22,59.01)}, rotate = 180] [fill={rgb, 255:red, 0; green, 0; blue, 0 }  ][line width=0.08]  [draw opacity=0] (10.72,-5.15) -- (0,0) -- (10.72,5.15) -- (7.12,0) -- cycle    ;
\draw    (413,71) .. controls (430.58,101.89) and (479.62,91.79) .. (492.02,70.38) ;
\draw [shift={(493.22,68.01)}, rotate = 113.5] [fill={rgb, 255:red, 0; green, 0; blue, 0 }  ][line width=0.08]  [draw opacity=0] (10.72,-5.15) -- (0,0) -- (10.72,5.15) -- (7.12,0) -- cycle    ;
\draw    (393.22,187.01) -- (357.22,187.01) ;
\draw [shift={(354.22,187.01)}, rotate = 360] [fill={rgb, 255:red, 0; green, 0; blue, 0 }  ][line width=0.08]  [draw opacity=0] (10.72,-5.15) -- (0,0) -- (10.72,5.15) -- (7.12,0) -- cycle    ;
\draw    (411,176) .. controls (434.25,136.65) and (478.5,150.74) .. (488.19,170.51) ;
\draw [shift={(489.22,173.01)}, rotate = 251.57] [fill={rgb, 255:red, 0; green, 0; blue, 0 }  ][line width=0.08]  [draw opacity=0] (10.72,-5.15) -- (0,0) -- (10.72,5.15) -- (7.12,0) -- cycle    ;
\draw    (418.22,187.01) -- (478.22,187.01) ;
\draw [shift={(481.22,187.01)}, rotate = 180] [fill={rgb, 255:red, 0; green, 0; blue, 0 }  ][line width=0.08]  [draw opacity=0] (10.72,-5.15) -- (0,0) -- (10.72,5.15) -- (7.12,0) -- cycle    ;
\draw    (412,199) .. controls (429.58,229.89) and (478.62,219.79) .. (491.02,198.38) ;
\draw [shift={(492.22,196.01)}, rotate = 113.5] [fill={rgb, 255:red, 0; green, 0; blue, 0 }  ][line width=0.08]  [draw opacity=0] (10.72,-5.15) -- (0,0) -- (10.72,5.15) -- (7.12,0) -- cycle    ;

\draw (91,103.4) node [anchor=north west][inner sep=0.75pt]    {$1$};
\draw (143,103.4) node [anchor=north west][inner sep=0.75pt]    {$2$};
\draw (196,103.4) node [anchor=north west][inner sep=0.75pt]    {$3$};
\draw (107,147) node [anchor=north west][inner sep=0.75pt]   [align=left] {A quiver $\displaystyle Q$};
\draw (337,50.4) node [anchor=north west][inner sep=0.75pt]    {$\overline{\Bbbk}$};
\draw (401,49.4) node [anchor=north west][inner sep=0.75pt]    {$\overline{\Bbbk}^{10}$};
\draw (486,48.4) node [anchor=north west][inner sep=0.75pt]    {$\overline{\Bbbk}^{6}$};
\draw (443,6.4) node [anchor=north west][inner sep=0.75pt]    {$\varphi _{0}$};
\draw (442,43.4) node [anchor=north west][inner sep=0.75pt]    {$\varphi _{1}$};
\draw (443,73.4) node [anchor=north west][inner sep=0.75pt]    {$\varphi _{2}$};
\draw (372,39.4) node [anchor=north west][inner sep=0.75pt]    {$f$};
\draw (272,101) node [anchor=north west][inner sep=0.75pt]   [align=left] {A representation $\displaystyle V$ of $\displaystyle Q$ with \textbf{dim}$\displaystyle ( V) =( 1,10,6)$};
\draw (336,178.4) node [anchor=north west][inner sep=0.75pt]    {$0$};
\draw (400,177.4) node [anchor=north west][inner sep=0.75pt]    {$\overline{\Bbbk}$};
\draw (485,177.4) node [anchor=north west][inner sep=0.75pt]    {$\overline{\Bbbk}$};
\draw (443,135.4) node [anchor=north west][inner sep=0.75pt]    {$\varphi _{0}$};
\draw (441,170.4) node [anchor=north west][inner sep=0.75pt]    {$\varphi _{1}$};
\draw (442,204.4) node [anchor=north west][inner sep=0.75pt]    {$\varphi _{2}$};
\draw (371,170.4) node [anchor=north west][inner sep=0.75pt]    {$0$};
\draw (250,226) node [anchor=north west][inner sep=0.75pt]   [align=left] {A subrepresentation $\displaystyle M$ of $\displaystyle V$ with \textbf{dim}$\displaystyle ( M) =( 0,1,1)$};

\end{tikzpicture}
\caption{The quiver key to Theorem \ref{thm: Reineke thm}}
    \label{fig: example of q grass}
\end{figure}

Given a quiver $Q$ and a representation $V$ of $Q$, the \textbf{quiver grassmannian} $\text{Gr}_{\bm e}^{Q} (V)$ is the set of subrepresentations of $V$ with dimension vector $\bm e$. The subrepresentation $M$ in Figure \ref{fig: example of q grass} is an element of $\text{Gr}_{(0,1,1)}^{Q}(V)$. It is well known that quiver grassmannians are projective varieties. For more on quiver grassmannians, see [\ref{ref: quiv grassmannians notes}]. We also have the following result of Hille, Huisgen-Zimmermann, and Reineke. The wording below is consistent with Reineke's in [\ref{ref: every projective variety is a}]: 

\begin{thm} \label{thm: Reineke thm}
Every projective variety is isomorphic to a quiver Grassmannian $\text{Gr}^Q_{\bm e}(V)$ for an acyclic quiver Q with at most three vertices, a Schurian representation V, and a thin dimension vector $\bm e$; that is, $e_i \leq 1$ for all $i \in Q_0$. \hfill $\square$
\end{thm} 

\section{The 4-fold Equivalence}

\indent

Theorem \ref{thm: Reineke thm} relies on the $d$-uple Veronese embedding. We will use essentially the same idea to create a category of quiver Grassmannians equivalent to the third category listed in Theorem \ref{thm: 3- fold Categorical equivalence}.\\

Let $X\subset \mathbb{P}_{\overline{\Bbbk}}^2$ be a smooth projective curve. Thus $X$ is defined as the vanishing locus of a homogeneous polynomial $P$ of degree $d$ in three variables. Let $\nu_d:\mathbb{P}_{\overline{\Bbbk}}^2 \rightarrow \mathbb{P}_{\overline{\Bbbk}}^{\binom{d+2}{2} - 1}$ denote the $d$-uple Veronese embedding, which is an isomorphism onto its image since $\overline{\Bbbk}$ is a field of characteristic 0. Then by Reineke's result, Theorem \ref{thm: Reineke thm}, $\nu_d(X) = \text{Gr}_{(0,1,1)}(V)$ for $V$ a representation of the quiver $Q$ in Figure \ref{fig: example of q grass} of dimension $(1,M,M')$ where $M = \binom{d+2}{2}$ and $M' = \binom{d+1}{2}$. On the top right of Figure \ref{fig: example of q grass} is an example of $\nu_3(X)$ where $X$ is a projective curve defined by the vanishing locus of a degree $d = 3$ polynomial in three variables. \\

The condition of being in the image of the $d$-uple Veronese embedding is encoded into an $M' \times 3$ matrix $A_d(x)$, whose $2\times 2$ sub minors vanish. In particular, let $M_{2,d}$ be the set of tuples $m = (m_0, m_1, m_2) \in \mathbb{N}^3$ summing to $d$, so that $M$ is the cardinality of $M_{2,d}$, and $\nu_d$ maps homogeneous coordinates $[x_0 , x_1, x_2]$ to $[\dots, x^m, \dots]_{m\in M_{2,d}}$, where $x^m = x_0^{m_0}x_1^{m_1}x_2^{m_2}$. We define the matrix $A_d(x)$ with rows indexed by the $n \in M_{2,d-1}$ and columns indexed by the $i = 0, 1, 2$ with the $(2,i)$-th entry being $x_{n+e_i}$. Then $x\in \nu_d(X)$ if and only if $P(x) = 0$ and $A_d(x)$ has rank 1. We thus define $V$ as in Figure \ref{fig: example of q grass} with $f = \nu_d(P)$ and $\varphi_i$ the $i$th column of $A_d(x)$. Then $\nu_d(X) = \text{Gr}_{(0,1,1)}(V)$. For more on this construction, see [\ref{ref: every projective variety is a}]. \\

Fix the quiver $Q$ to be the one in Figure \ref{fig: example of q grass} and let $d \in \mathbb{Z}^{\geq0}$. Let $V_d^f$ be a $\overline{\Bbbk}$-representation of $Q$ such that $\textbf{dim}(V) = (1,M,M')$, $\varphi_i$ is the $i$th column of $A_d(x)$, and the preimage of the linear map $f$, denoted by  $\nu_d^{-1}(f)$, is irreducible as a homogeneous polynomial in $\overline{\Bbbk}[x_0, x_1, x_2]$. \\

\begin{defn} \label{defn: cat of quiver grass}
Define  $\mathcal{GR}_{(0,1,1)}(V)$ to be the category whose objects are $\text{Gr}_{(0,1,1)}(V_d^f)$ for all $d \in \mathbb{Z}^{\geq0}$ and any $f$ such that $\nu_d^{-1}(f)$ is irreducible, and whose morphisms are those of projective varieties. Let $\mathcal{GR}_{(0,1,1)}^{\text{sm}}(V)$ be the full subcategory of  $\mathcal{GR}_{(0,1,1)}(V)$ who's objects are smooth. \\
\end{defn} 

\begin{thm}\label{thm: iso of cats}
The category of non-singular projective curves in $\mathbb{P}^2_{\overline{\Bbbk}}$ with morphisms of varieties, (NPC) for short, is isomorphic to $\mathcal{GR}_{(0,1,1)}^{\text{sm}}(V)$.
\end{thm}

\begin{proof}
We begin by constructing a functor $\nu: \text{(NPC)}\rightarrow \mathcal{GR}_{(0,1,1)}^{\text{sm}}(V)$. For $X$ a non-singular projective curve cut out by a homogeneous polynomial of degree $d$, define $\nu(X) := \nu_d(X)$. Then $\nu(X)$ is non-singular since $X$ is and $\nu(X) \cong X$. Moreover, $\nu(X) = \text{Gr}_{(0,1,1)}(V_d^f)$ for $f$ the homogeneous polynomial of degree $d$ that cuts out $X$ by Reineke's result, Theorem \ref{thm: Reineke thm}. Thus $\nu(X) \in \mathcal{GR}_{(0,1,1)}^{\text{sm}}(V_d)$ and $\nu$ is well defined on objects. Given a morphism $\psi: X\rightarrow Y$ between two non-singular projective curves cut out by homogeneous polynomials of degree $d$ and $d'$ respectively, define $\nu(\psi): \nu(X) \rightarrow \nu(Y)$ by $\nu_{d'} \circ \psi \circ \nu_{d}^{-1}$. Since $\nu(\psi)$ is a morphism of varieties and $\nu$ preserves the identity and composition, $\nu$ defines a functor. \\

It is well known that the $d$-uple Veronese embedding is an isomorphism onto its image. Notice by construction, for any $d$, the image of $\nu_d$ restricted to projective curves cut out by a homogeneous polynomial of degree $d$ is equal to the collection of $\text{Gr}_{(0,1,1)}(V_d^f)$. This allows us to define $\nu^{-1}: \mathcal{GR}_{(0,1,1)}^{\text{sm}}(V) \rightarrow (\text{NPC})$ analogously to $\nu$, and these two functors are inverse. \end{proof}

\noindent

By taking $\Bbbk = \mathbb{C}$, an immediate consequence of Theorems \ref{thm: Reineke thm} and \ref{thm: iso of cats} is the following 4-fold categorical equivalence. 

\begin{cor}\label{cor: 4-fold eq}
The following four categories are equivalent:
\begin{enumerate}
    \item Compact connected Riemann Surfaces with holomorphic maps.
    \item Field extensions of transcendence degree one over $\mathbb{C}$ with field morphisms.
    \item Smooth projective curves in $\mathbb{P}^2_{\overline{\mathbb{C}}}$ with morphisms of varieties.
    \item $\mathcal{GR}_{(0,1,1)}^{\text{sm}}(V)$ where $V$ is a $\mathbb{C}$-representation of $Q$. \hfill $\square$
\end{enumerate}
\end{cor}

\vspace{.3cm}
\noindent 

Recall that by definition, elliptic curves are non-singular curves of genus one; however, every such curve can be written as the locus in $\mathbb{P}_{\overline{\Bbbk}}^2$ of a cubic equation with the base point on the line at $\infty$ [\ref{ref: elliptic curves reference}]. The next corollary follows from the fact that elliptic curves are the vanishing locus of an irreducible homogeneous polynomial of degree 3 and the fact that the functor $\nu$ restricts to an isomorphism, namely $\nu_3$. \\

\begin{cor} \label{cor: elliptic curves equiv to quiver grassmannians}
The category of elliptic curves in $\mathbb{P}_{\overline{\Bbbk}}^2$ with morhpisms of varieties is equivalent to $\mathcal{GR}_{(0,1,1)}^{\text{sm}}(V_3)$, the full subcategory of $\mathcal{GR}_{(0,1,1)}^{\text{sm}}(V)$ whose objects are $\text{Gr}_{(0,1,1)}^{\text{sm}}(V_3^f)$ for $f$ such that $\nu_3^{-1}(f)$ is irreducible. \qed
\end{cor}

\begin{rem}
Since in Corollary \ref{cor: elliptic curves equiv to quiver grassmannians} we do not take $\Bbbk = \mathbb{C}$, this equivalence along with the linear-algebraic nature of representations of quivers may prove to be useful in the study of rational points of elliptic curves. Moreover, one may be able to use the moduli space of quiver grassmannians to study that of elliptic curves and vice versa. \\
\end{rem}

Corollary \ref{cor: elliptic curves equiv to quiver grassmannians} also provides us with a way to remove the artificial imposition of smoothness in Definition \ref{defn: cat of quiver grass}. In determining smoothness of $\text{Gr}_{(0,1,1)}(V_d^f)$, it suffices to check the Jacobian criterion on the equations that cut out the quiver grassmannian; however, after embedding into a higher dimensional projective space there can be several of these equations and checking this criterion can quickly become computationally expensive. We do however have the following proposition.

\begin{prop}
Suppose the characteristic of $\Bbbk$ is not 2 or 3 and consider a quiver grassmannian $\text{GR}_{(0,1,1)}(V_3^f)$. Then  $\text{GR}_{(0,1,1)}(V_3^f)$ is smooth if and only if it is isomorphic to $\text{GR}_{(0,1,1)}(V_3^{\xi})$ where $\xi = x_7 - x_0 - ax_5 - bx_9$ and $4a^3 + 27b^2 \neq 0$.
\end{prop}

\begin{proof}
By definition, a curve in $\mathbb{P}_{\overline{\Bbbk}}^2$ cut out by a degree 3 homogeneous polynomial is smooth if and only if it is an elliptic curve. In the case of elliptic curves however, it is known that each curve can be written in reduced Weierstrass form as $y^2 = x^3 + ax + b$ when the field is not characteristic 2 or 3 [\ref{ref: elliptic curves reference}]. Upon realizing this in homogeneous coordinates, we get the equation $y^2z - x^3- axz^2 - bz^3 = 0$. To attain the corresponding quiver grassmannian, we analyze the 3-uple Veronese embedding:
\vspace{-.1 cm}
$$\nu_3(x,y,z) = (x^3,x^2y,x^2z,xy^2,xyz,xz^2,y^3,y^2z,yz^2,z^3).$$
\vspace{.08 cm}
Relabeling the $\mathbb{P}_{\overline{\Bbbk}}^9$ coordinates as $(x_0, x_1, \dots , x_9)$, any elliptic curve is the solution set to $x_7 - x_0 - ax_5 - bx_9 = 0$. Letting $\xi = x_7 - x_0 - ax_5 - bx_9$, the corresponding quiver grassmannian is $\text{Gr}_{(0,1,1)}(V_3^{\xi})$. Now by Corollary \ref{cor: elliptic curves equiv to quiver grassmannians}, we have that $\text{Gr}_{(0,1,1)}(V_3^{\xi})$ is an object of $\mathcal{GR}_{(0,1,1)}^{\text{sm}}(V_3)$ if and only if $\xi$ is smooth, which occurs if and only if $4a^3 + 27b^2 \neq 0$ [\ref{ref: elliptic curves reference}]. 
\end{proof}

Using Corollary \ref{cor: elliptic curves equiv to quiver grassmannians} we can also see that the objects of $\text{GR}_{(0,1,1)}^{\text{sm}}(V_3^f)$ can be endowed with a commutative group structure inherited from that of elliptic curves. In the case $\Bbbk = \mathbb{C}$, we can use Corollary \ref{cor: 4-fold eq} to further state that these quiver grassmannians are also isomorphic to connected compact Riemann surfaces of genus 1, hence complex tori. \\

\begin{cor}
Let $X$ be a connected compact Riemann surface. Then the following are equivalent:
\begin{enumerate}
    \item $X$ has genus 1.
    \item $X$ has a structure of an analytic group.
    \item $X$ has a commutative analytic group structure.
    \item $X\cong {\mathbb{C} \over \mathbb{Z}l + \mathbb{Z}w}$
    \item $X \cong C_F$ where $C_F$ is an elliptic curve. 
    \item $X\cong \text{Gr}_{(0,1,1)}^{\text{sm}}(V_3^f)$ for some $f. \hfill \square$
\end{enumerate}
\end{cor}

\section{References}

\begin{enumerate}[ {[}1{]} ]
\item Cerulli Irelli, G. \textit{Three lectures on quiver Grassmannians: Representation theory and beyond}, Contemp. Math., vol. 758, Amer. Math. Soc., [Providence], RI, 2020, pp. 57 – 89, DOI 10.1090/conm/758/15232. MR4186968 \label{ref: quiv grassmannians notes} 

\item Hartshorne, R. \textit{Algebraic geometry}, Springer, Volume 52, 1977.\label{ref: harshorne}

\item Hille, L. \textit{Moduli of representations, quiver Grassmannians and Hilbert schemes}, Preprint, arXiv:1505.06008, [math.RT], 2015. \label{ref: Hille generalized Kronecker}

\item Huisgen-Zimmermann, B. \textit{Classifying representations by way of Grassmannians}, Trans. Amer. Math. Soc. 359 (2007), no. 6, 2687–2719, DOI 10.1090/S0002-9947-07-03997-9. MR2286052 \label{ref: H-Z}

\item Reineke, M. \textit{Every projective variety is a quiver Grassmannian}, Algebras and Representation Theory, \textbf{16}, 1313-1314, 2013. \label{ref: every projective variety is a}

\item Ringel, C. \textit{The eigenvector variety of a matrix pencil}, Linear Algebra Appl., \textbf{531} (2017), 447-458. \label{ref: generalized Kronecker}

\item Ringel, C. \textit{Quiver Grassmannians and Auslander varieties for wild algebras}, J. of Alg., Volume 402, 351-357, 2014 \label{ref: ringel Auslander varieties}

\item Ringel, C. \textit{Quiver Grassmannians for wild acyclic quivers}, Proc. AMS, \textbf{146}, n. 5, 2018. \label{ref: generalizaion of Reineke result}

\item Schofield, A. \textit{General representations of quivers}, Proc. London Math. Soc. (3) 65 (1992), no. 1, 46–64.\label{ref: quiver grassmannians intro}

\item Silverman, Joseph H. \textit{The arithmetic of elliptic curves}, Springer, Volume 106, 2009. \label{ref: elliptic curves reference}
\end{enumerate}

\end{document}